\documentclass[12pt]{amsart}
\usepackage{amsmath,amsfonts,amsbsy,amsgen,amscd,mathrsfs,amssymb,amsthm}
\usepackage[usenames,dvipsnames]{xcolor}
\usepackage[colorlinks=true,citecolor=red,linkcolor=blue]{hyperref}
\usepackage{geometry}
\usepackage{setspace}
\allowdisplaybreaks[4]
\numberwithin{equation}{section}
\newtheorem{theorem}{Theorem}[section]
\newtheorem{definition}[theorem]{Definition}
\newtheorem{lemma}[theorem]{Lemma}
\newtheorem{corollary}[theorem]{Corollary}

\newtheorem{proposition}[theorem]{Proposition}

\newtheorem{problem}[theorem]{Problem}

\title[$L_p$ Gauss dual Minkowski problem]
{The $L_p$ Gauss dual Minkowski problem}

\author{Na Fu}
\address{Na Fu \newline \indent School of Science, Lanzhou University of Technology, Lanzhou, 730050, China}
\curraddr{}
\email{funa930323@163.com}
\thanks{}

\author{Jianping Sun}
\address{Jianping Sun \newline \indent School of Science, Lanzhou University of Technology, Lanzhou, 730050, China}
\curraddr{}
\email{jpsun@lut.edu.cn}
\thanks{Research is supported by the National Natural Science Foundation of China (No. 11661049)}
\thanks{Corresponding author: Jianping Sun}

\subjclass[2020]{52A20; 52A40; 35K96}

\keywords{$L_p$-Gauss dual Minkowski problem; Monge-Amp\`{e}re equation; Gauss curvature flow}

\begin{document}
\begin{abstract}
  This article introduces the $L_p$-Gauss dual curvature measure and proposes its related $L_p$-Gauss dual Minkowski problem as: for $p,q\in\mathbb{R}$,
  under what necessary and/or sufficient condition on a non-zero finite Borel measure $\mu$ on unit sphere does there exist a convex body $K$ such that $\mu$ is the $L_p$ Gauss dual curvature measure? If $K$ exists, to what extent is it unique?
  This problem amounts to solving a class of Monge-Amp\`{e}re type equations on unit sphere in smooth case:
  \begin{align}\label{0.1}
  e^{-\frac{|\nabla h_K|^2+h_K^2}{2}}h_K^{1-p}
  (|\nabla h_K|^2+h_K^2)^{\frac{q-n}{2}}
  \det(\nabla^2h_K+h_KI)=f,
  \end{align}
  where $f$ is a given positive smooth function on unit sphere, $h_K$ is the support function of convex body $K$, $\nabla h_K$ and $\nabla^2h_K$ are the gradient and Hessian of $h_K$ on unit sphere with respect to an orthonormal basis, and $I$ is the identity matrix.

  We confirm the existence of convex body solution to the new problem with $p,q>0$ and the existence of smooth solution to the equation (\ref{0.1}) with $p ,q\in\mathbb{R}$ by variational method and Gaussian curvature flow method, respectively.
  Furthermore, the uniqueness of solution to the equation (\ref{0.1}) in the case $p,q\in\mathbb{R}$ with $q<p$ is established.
\end{abstract}

\maketitle

\vskip 20pt
\section{Introduction }
The Brunn-Minkowski type theories (such as classical Brunn-Minkowski theory, $L_p$-Brunn-Minkowski theory, Orlicz Brunn-Minkowski theory and their dual theories) of convex bodies (compact convex sets with nonempty interiors) in $n$-dimensional Euclidean spaces $\mathbb{R}^n$ play an important role in the study of convex geometric analysis.
Geometric invariants, Minkowski sum and geometric measures associated with convex bodies are the core of the Brunn-Minkowski type theories. Geometric invariants can be viewed as geometric functionals of convex bodies and geometric measures are the differential of geometric functionals. It is well known that Minkowski-type problems related to geometric measures (including surface area measures, curvature measures and dual curvature measures, etc.) are the cornerstone of Brunn-Minkowski theories, see e.g., \cite{BB,BH,BoL,CFL,CHZ,CHZ1,CW,G,GHW,GaHW,HaLYZ,HLYZ,
HuLY,HLY,JW,L,Lu,LYZ,S,Z} and the references therein. In the last decade, the Minkowski type problems for the measures associated with solution to the boundary-value problems are doubtless extremely important variant, as some typical examples, we can refer to some representative papers \cite{CF,CNS,FZH,HYZ,J,ZX}.
Moreover, a particularly interesting Minkowski type problem, the Gauss Minkowski problem, was first proposed by Huang, Xi and Zhao \cite{HXZ}, and more research results on this problem can be found in \cite{BLYZ,FHX,FLX,HQ,Liu,W,WW}.

Very recently, Feng, Li and Xu \cite{F} devepoped a theory analogous to the one for the Minkowski problem in which geometric measure are generated by the differential of {\it Gauss dual quermassintegral}. Let $K\in\mathcal{K}_o^n$  be the set of convex bodies containing the origin in their interior. For $q\in\mathbb{R}$ and $K\in\mathcal{K}_o^n$, the Gauss dual quermassintegral, denoted by  $\widetilde{V}_{\gamma,q}(K)$, of $K$ is defined by
\begin{align}\label{1.1}
\widetilde{V}_{\gamma,q}(K)=\int_{\mathbb{R}^n\backslash K}
e^{-\frac{|x|^2}{2}}|x|^{q-n}dx.
\end{align}
The Gauss dual curvature measure, denoted by $\widetilde{C}_{\gamma,q}(K,\cdot)$, of $K\in\mathcal{K}_o^n$ is uniquely determined by the following variational formula (see \cite[Thorem 3.3]{F})
$$\lim_{t\rightarrow0}\frac{\widetilde{V}_{\gamma,q}([h_t])
-\widetilde{V}_{\gamma,q}(K)}{t}
=-\int_{\mathbb{S}^{n-1}}f(v)
d\widetilde{C}_{\gamma,q}(K,v),$$
where $h_t=e^{tf}h_K$ and a continuous function $f: \mathbb{S}^{n-1}\rightarrow\mathbb{R}$, and $[h_t]$ is the Wulff shape generated by $h_t$ (see Section \ref{S2} for details). For each Borel measurable set $\omega\subset\mathbb{S}^{n-1}$, the Gauss dual curvature measure $\widetilde{C}_{\gamma,q}(K,\cdot)$ is defined by
$$\widetilde{C}_{\gamma,q}(K,\omega)=
\int_{\nu_K^{-1}(\omega)}x\cdot\nu_K(x)e^{-\frac{|x|^2}{2}}
|x|^{q-n}d\mathcal{H}^{n-1}(x).$$
where $\cdot$ represents the inner product in $\mathbb{R}^n$.

The {\bf Gauss dual Minkowski problem} was posed: given a non-zero finite Borel measure $\mu$ on $\mathbb{S}^{n-1}$ and $q\in\mathbb{R}$, what are necessary and sufficient conditions on $\mu$ to
guarantee the existence of a convex body $K\in\mathcal{K}_o^n$ such that
$$\mu=\widetilde{C}_{\gamma,q}(K,\cdot)?$$
If $K$ exists, is it unique?

In \cite{F}, the existence of smooth solutions to Gauss dual Minkowski problem for $q<0$ was obtained by continuity method. Then, an approximation argument was used to give the weak solution to this problem. Moreover, the existence of smooth for $q=0$ was established by a degree-theory approach,  and the uniqueness of smooth solution for $q\leq0$ was also established.

The main purpose of this paper is to construct $L_p$-Gauss dual curvature measure based on the research in \cite{F}.
An $L_p$ version of the variational formula similar to \cite[Theorem 6.2]{LYZ} can be shown: for $p\neq0, q\in\mathbb{R}$, $K\in\mathcal{K}_o^n$ and a continuous function $f :\mathbb{S}^{n-1}\rightarrow\mathbb{R}$, then
$$\lim_{t\rightarrow0}\frac{\widetilde{V}_{\gamma,q}([h_t])
-\widetilde{V}_{\gamma,q}(K)}{t}
=-\frac{1}{p}\int_{\mathbb{S}^{n-1}}f(v)^p
d\widetilde{C}_{\gamma,p,q}(K,v),$$
where $h_t(v)=(h_K(v)^p+tf(v)^p)^{\frac{1}{p}}$. This naturally leads to the definition of $L_p$ Gauss dual curvature measure:
$$\widetilde{C}_{\gamma,p,q}(K,\omega)
=\int_{\nu_K^{-1}(\omega)}(x\cdot\nu_K(x))^{1-p}
e^{-\frac{|x|^2}{2}}|x|^{q-n}d\mathcal{H}^{n-1}(x),$$
for Borel measurable set $\omega\subset\mathbb{S}^{n-1}$.

Obviously, the $L_p$-Gauss surface area measure $S_{\gamma,p}(K,\cdot)$ and the Gauss dual curvature measure $\widetilde{C}_{\gamma,q}(K,\cdot)$ are special cases of the $L_p$-Gauss dual curvature measure in the sense that for $p, q\in\mathbb{R}$ and $K\in\mathcal{K}_o^n$,
$$\widetilde{C}_{\gamma,0,q}(K,\cdot)
=\widetilde{C}_{\gamma,q}(K,\cdot), \ \ \widetilde{C}_{\gamma,p,n}(K,\cdot)
=S_{\gamma,p}(K,\cdot). $$

Naturally, the $L_p$-Gauss dual Minkowski problem is posed as follows:
\begin{problem}\label{p1.1}
(The $L_p$-Gauss dual Minkowski problem)
For $p,q\in\mathbb{R}$,
under what necessary and/or sufficient condition on a non-zero finite Borel measure $\mu$ on $\mathbb{S}^{n-1}$ does there exist a set $K\in\mathcal{K}_o^n$ such that
\begin{align*}
\mu=\widetilde{C}_{\gamma,p,q}(K,\cdot)?
\end{align*}
If $K$ exists, to what extent is it unique$?$
\end{problem}

When $K$ is sufficiently smooth and the given measure $\mu$ has a density $f: \mathbb{S}^{n-1}\rightarrow \mathbb{R}$, the $L_p$-Gauss dual Minkowski problem is equivalent to solving the following Monge-Amp\`{e}re type equation on $\mathbb{S}^{n-1}$:
\begin{align}\label{1.2}
e^{-\frac{|\nabla h_K|^2+h_K^2}{2}}h_K^{1-p}
(|\nabla h_K|^2+h_K^2)^{\frac{q-n}{2}}
\det(\nabla^2h_K+h_KI)=f,
\end{align}
where $\nabla h_K$ and $\nabla^2h_K$ are the gradient and Hessian of $h_K$ on $\mathbb{S}^{n-1}$ with respect to an orthonormal basis, and $I$ is the identity matrix.

The first objective is to consider the existence of symmetric solution to Problem \ref{p1.1} by the variational method (see Theorem \ref{t4.3} for details).

\begin{theorem}\label{t1.2}
Let $\mu$ be a non-zero finite even Borel measure on $\mathbb{S}^{n-1}$ and not concentrated in any closed hemisphere. If $p>0$ and $q>0$, then there exists an origin-symmetric convex body $K$ such that
\begin{align*}
\mu=\widetilde{C}_{\gamma,p,q}(K,\cdot).
\end{align*}
\end{theorem}

Next, using the Gauss curvature flow method, the existence of smooth solution to the equation (\ref{1.2}) is obtained as below (see Theorem \ref{t5.5} for details).

\begin{theorem}\label{t1.3}
Let {$p, q\in\mathbb{R}$}, and $f: \mathbb{S}^{n-1}\rightarrow(0,\infty)$ be a smooth, positive function satisfying
\begin{align}\label{1.3}
\limsup_{s\rightarrow+\infty}
(s^{q-p}e^{-\frac{s^2}{2}})<f<\liminf_{s\rightarrow0^+}
(s^{q-p}e^{-\frac{s^2}{2}}).
\end{align}
Then there exists a smooth solution $h_K$ to the equation (\ref{1.2}).
\end{theorem}

The Gauss curvature flow was first introduced and studied by Firey \cite{Fi} to model the shape change of worn stones.  Since then, many scholars have found that using curvature flow to study the hypersurfaces is a very effective tool, such as solving the Minkowski-type problems, see e.g., \cite{BIS,CC,CH,CL,CWX,LSW,LL} and the references therein.

Finally, we establish the uniqueness of the solution to the equation (\ref{1.2}) as follows (see Theorem \ref{t6.1} for details).

\begin{theorem}\label{t1.4}
Let $p, q\in\mathbb{R}$ with $q<p$, and $f: \mathbb{S}^{n-1}\rightarrow(0,\infty)$ be a smooth, positive function on $\mathbb{S}^{n-1}$. Then the solution to the equation (\ref{1.2}) is unique.
\end{theorem}

This paper is organized as follows.
In Section \ref{S2}, we collect some basic facts on convex bodies and convex hypersurfaces.
In Section \ref{S3}, we establish a $L_p$-variational formula of Gauss dual quermassintegral, and introduce the $L_p$-Gauss dual curvature measure and $L_p$-Gauss dual Minkowski problem.
In Section \ref{S4}, the existence of convex body solution to the $L_p$-Gauss dual Minkowski problem will be discussed by variational method.
In Section \ref{S5}, the existence of smooth solution to the $L_p$-Gauss dual Minkowski problem will be obtained by curvature flow method.
Finally, the uniqueness of the solution of $L_p$-Gauss dual Minkowski problem will be established in Section \ref{S6}.

\section{Preliminaries}\label{S2}
In this section we list some facts about convex bodies and convex hypersurfaces that readers can refer to two celebrated books of Gardner and Schneider \cite{G,S} and papers \cite{HLYZ,U}.

For a convex body $K\in\mathcal{K}_o^n$, the support function, $h(K,\cdot): \mathbb{S}^{n-1}\rightarrow\mathbb{R}$, of $K$ is defined by
$$h_K(u)=h(K,u)=\max\{u\cdot Y: Y\in K\}, \ u\in\mathbb{S}^{n-1}.$$
We clearly know that the support function is a continuous and homogeneous convex function with degree $1$. The Minkowski sum of $K, L\in\mathcal{K}_o^n$ is defined, for $t>0$, by
$$K+tL=\{x+ty: x\in K, \ y\in L\}.$$
The set $K+tL$ is a convex body whose support function is given by
$$h(K+tL,\cdot)=h(K,\cdot)+th(L,\cdot).$$
For $K, L\in\mathcal{K}_o^n$ and $p\geq1$, the $L_p$ Minkowski sum, $K+_pt\cdot L$, of $K$ and $L$ is defined through its support function
$$h(K+_pt\cdot L,\cdot)^p=h(K,\cdot)^p+th(L,\cdot)^p.$$
In particular, when $p=1$, the $L_p$ Minkowski sum is just classical Minkowski sum.

For each $K\in\mathcal{K}_o^n$, the radial function $\rho(K,\cdot)$ is defined by
$$\rho_K(x)=\rho(K,x)=\max\{\lambda: \lambda x\in K\},\ x\in\mathbb{R}^n\backslash\{0\}.$$
Obviously, the radial function is a continuous function and homogeneous of degree $-1$. Moreover $\partial K=\{\rho_K(v)v: v\in\mathbb{S}^{n-1}\}$.

For each $K\in\mathcal{K}_o^n$, the polar body $K^*$ of $K$ is the convex body defined by
$$K^*=\{x\in\mathbb{R}^n: x\cdot y\leq1, \ for\ all\ y\in K\}.$$
It is easy to verify that
\begin{align}\label{2.1}
(K^*)^*=K.
\end{align}
The support function and radial function of $K\in\mathcal{K}_o^n$ and its polar body are related by
\begin{align}\label{2.2}
h_K(x)\rho_{K^*}(x)=1 \ \ and \ \ \rho_{K}(x)h_{K^*}(x)=1,
\end{align}
for $x\in\mathbb{R}^n\backslash\{0\}$.

Let Let $C(\mathbb{S}^{n-1})$ be the set of all continuous functions on $\mathbb{S}^{n-1}$,  $C^+(\mathbb{S}^{n-1})$ be the set of all strictly positive continuous functions on $\mathbb{S}^{n-1}$ and $C^+_e(\mathbb{S}^{n-1})$ be the subset of even functions from $C^+(\mathbb{S}^{n-1})$. For each $f\in C^+(\mathbb{S}^{n-1})$, the Wulff shape $[f]$ generated by $f$ is the convex body defined by
$$[f]=\{x\in\mathbb{R}^n: x\cdot v\leq f(v),\ for \ all \ v\in\mathbb{S}^{n-1}\}.$$
It is apparent that $h_{[f]}\leq f$ and $[h_K]=K$ for $K\in\mathcal{K}_o^n$. The convex hull $\langle f\rangle$ generated by $f$ is defined by
$$\langle f\rangle=conv\{f(u)u: u\in\mathbb{S}^{n-1}\}\in\mathcal{K}_o^n.$$
If $K\in\mathcal{K}_o^n$, then $\langle \rho_K\rangle=K$. Moreover,
\begin{align}\label{2.3}
[f]^*=\langle\frac{1}{f}\rangle.
\end{align}

The boundary point of convex body $K$ which only has one supporting hyperplane is called the singular point. The set of singular points is denoted by $\sigma K$, it is
well known that $\sigma K$ has spherical Lebesgue measure 0.
For $x\in\partial K\setminus \sigma K$, the Gauss map $\nu_K:x\in\partial K\setminus \sigma K\rightarrow \mathbb{S}^{n-1}$ is represented by
\begin{align*}
\nu_K(x)=\{v\in\mathbb{S}^{n-1}:x\cdot v=h_K(v)\}.
\end{align*}
Correspondingly, for a Borel set $\eta\subset\mathbb{S}^{n-1}$, the inverse Gauss map is denoted by
$\nu_K^{-1}$,
\begin{align*}
\nu_K^{-1}(\eta)=\{x\in\partial K: \nu_K(x)\in\eta\}.
\end{align*}

Note that $h_K(x)$ is differentiable at $x\in\mathbb{R}^n$ if and only if $\partial h_K(x)$ consists of exactly one vector which is the gradient, denoted by $Dh_K(x)$, of $h_K$ at $x$.
Let $h_K$ be differentiable at $v\in\mathbb{S}^{n-1}$ for a convex body $K$, where $v=\nu_K(x)$ is an outer unit normal vector at $x\in\partial K$. Then
$$x=\nu_K^{-1}(v)=Dh_K(v).$$
From this, we have
\begin{align}\label{2.5}
&\nonumber h_K(v)=h_K(\nu_K(x))=\langle x,\nu_K(x)\rangle=\langle Dh_K(v),v\rangle,\\
&Dh_K(v)=\nabla h_K(v)+h_K(v)v,\\
&\nonumber |x|^2=|Dh_K(v)|^2=|\nabla h_K(v)|^2+h_K(v)^2.
\end{align}

For a convex hypersurface $\Omega$ of class $C^2_+$ ($\partial\Omega$ is $C^2$ smooth and has positive Gauss curvature), then the support function of $\Omega$ can be stated as
$$h(\Omega,x)=x\cdot\nu^{-1}_\Omega(x).$$
Let $e=\{e_{ij}\}$ be the standard metric of $\mathbb{S}^{n-1}$.
The second fundamental form of $\Omega$ is defined as
\begin{align}\label{2.5}
\Pi_{ij}=\nabla_{ij}h+he_{ij},
\end{align}
where $\nabla_{ij}$ is the second order covariant derivative with respect to $e_{ij}$.
By the Weingarten's formula and (\ref{2.5}), the principal radii of $\Omega$, under a smooth local orthonormal frame on $\mathbb{S}^{n-1}$, are the eigenvalues of matrix
\begin{align*}
b_{ij}=\nabla_{ij}h+h\delta_{ij}.
\end{align*}
Particularly, the Gauss curvature of $\Omega$ can be expressed as
\begin{align*}
\mathcal{K}(x)=\frac{1}{\det(\nabla_{ij}h+h\delta_{ij})}.
\end{align*}

\section{$L_p$-Gaussian dual curvature measure and $L_p$ Gaussian dual Minkowski problem}\label{S3}
In this section, we establish the $L_p$ variational formula of Gauss dual quermassintegral, and introduce the $L_p$ Gauss dual curvature measure and its related $L_p$ Gauss dual Minkowski problem.

\begin{lemma}\label{l3.1}
Let $K\in\mathcal{K}_o^n$, $f: \mathbb{S}^{n-1}\rightarrow\mathbb{R}$ be a continuous function, and $\varepsilon>0$ be sufficiently small. Define $h_t\in C^+(\omega)$, for $t\in(-\varepsilon,\varepsilon)$ and $v\in\mathbb{S}^{n-1}$, by
$$h_t(v)=(h_K(v)^p+tf(v)^p)^{\frac{1}{p}}.$$
Then, for almost all $u\in\mathbb{S}^{n-1}$ with respect to the spherical Lebesgue measure,
$$\lim_{t\rightarrow0}\frac{\rho_{[h_t]}(u)
-\rho_K(u)}{t}=\frac{f(\alpha_K(u))^p}
{ph_K(\alpha_K(u))^p}\rho_K(u).$$
Moreover, there exists $M_0>0$ and $\varepsilon_0>0$ such that
$$|\rho_{[h_t]}(u)-\rho_K(u)|\leq M_0|t|,\ t\in(-\varepsilon_0,\varepsilon_0).$$
\end{lemma}

\begin{proof}
From \cite[Lemma 4.1]{HLYZ}, we have, for $u\in\mathbb{S}^{n-1}$ and $t\in(-\varepsilon,\varepsilon)$
\begin{align*}
&\frac{d}{dt}(\log h_{\langle\varrho_t\rangle}(u))|_{t=0}
\lim_{t=0}\frac{h_{\langle\varrho_t\rangle}(u)-
h_{\langle\varrho_0\rangle}(u)}{t}\\
&=\frac{1}{h_{\langle\varrho_t\rangle}(u)}\bigg|_{t=0}
\lim_{t=0}\frac{h_{\langle\varrho_t\rangle}(u)-
h_{\langle\varrho_0\rangle}(u)}{t}
=g(\alpha^\ast_{\langle\varrho_0\rangle}(u)),
\end{align*}
i.e.,
\begin{align}\label{3.1}
\lim_{t=0}\frac{h_{\langle\varrho_t\rangle}(u)-
h_{\langle\varrho_0\rangle}(u)}{t}
=h_{\langle\varrho_t\rangle}(u)
g(\alpha^\ast_{\langle\varrho_0\rangle}(u)),
\end{align}
where $g: \mathbb{S}^{n-1}\rightarrow \mathbb{R}$ is continuous and $\varrho_t: \mathbb{S}^{n-1}\rightarrow(0,\infty)$ is defined by
$$\log\varrho_t(u)=\log\varrho_0(u)+tg(u)+o(t,u).$$

By (\ref{2.2}) and (\ref{2.3}), one has
\begin{align}\label{3.2}
\nonumber&\lim_{t\rightarrow0}\frac{\rho_{[h_t]}(u)
-\rho_K(u)}{t}=\lim_{t\rightarrow0}
\frac{h_{[h_t]^*}^{-1}(u)
-h_{K^*}^{-1}(u)}{t}\\
\nonumber&=-\lim_{t\rightarrow0}\frac{h_{
\langle\frac{1}{h_t}\rangle}(u)
-h_{\langle\frac{1}{h_K}\rangle}(u)}
{th_{\langle\frac{1}{h_t}\rangle}(u)
h_{\langle\frac{1}{h_K}\rangle}(u)}\\
\nonumber&=-\frac{1}{h_{
\langle\frac{1}{h_K}\rangle}(u)^2}
\lim_{t\rightarrow0}\frac{h_{\langle\frac{1}{h_t}\rangle}(u)
-h_{\langle\frac{1}{h_K}\rangle}(u)}{t}\\
&=-\rho_K^2(u)\lim_{t\rightarrow0}
\frac{h_{\langle\frac{1}{h_t}\rangle}(u)
-h_{\langle\frac{1}{h_0}\rangle}(u)}{t}.
\end{align}

Let $\varrho_t=\frac{1}{h_t}$. Then
$$\log h_t=\log h_K-tg-o(t).$$
Since
$$h_t(v)=(h_K(v)^p+tf(v)^p)^{\frac{1}{p}},\ v\in\mathbb{S}^{n-1},$$
it follows that
$$\log h_t=\log h_K+\frac{tf^p}{ph_K^p}+o(t).$$
Thus
\begin{align*}
g=-\frac{f^p}{ph_K^p}.
\end{align*}
By (\ref{2.1}), (\ref{3.1}), (\ref{3.2}) and \cite[Lemma 2.6]{HLYZ}, we have
\begin{align*}
\lim_{t\rightarrow0}\frac{\rho_{[h_t]}(u)
-\rho_K(u)}{t}&=\frac{1}{p}\rho_K^2(u)
h_{\langle\frac{1}{h_0}\rangle}(u)
g(\alpha^\ast_{\langle\varrho_0\rangle}(u))\\
&=\frac{f(\alpha_K(u))^p}{ph_K(\alpha_K(u))^p}\rho_K(u).
\end{align*}

Since $[h_0]=K\in\mathcal{K}_o^n$ and $[h_t]\rightarrow K$ as $t\rightarrow0$ by Aleksandrov's convergence theorem for Wulff shapes (see \cite{S}), then there exist $m_1, m_2\in(0,\infty)$ and $\varepsilon_0>0$ such that
$$0<m_1<\rho_{[h_t]}<m_2<\infty \ \ on \ \ \mathbb{S}^{n-1},$$
for $t\in(-\varepsilon_0,\varepsilon_0)$. Thus
\begin{align}\label{3.3}
0<\frac{\rho_{[h_t]}}{\rho_{K}}<m_3 \ \ on \ \  \mathbb{S}^{n-1},
\end{align}
for some $m_3>1$. Observe that $s-1\geq\log s$ for $s\in(0,1)$ while $s-1\leq\log s\leq m_3\log s$ for $s\in[1,m_3]$. Thus, when $s\in(0,m_3)$
$$|s-1|\leq m_3|\log s|.$$
By (\ref{3.3}), we have
$$\bigg|\frac{\rho_{[h_t]}}{\rho_{K}}-1\bigg|\leq m_3\bigg|\log\frac{\rho_{[h_t]}}{\rho_{K}}\bigg|.$$
This, together with (\ref{3.3}) and fact that $|\log\rho_{[h_t]}-\log\rho_{K}|\leq M|t|$ (by \cite[Lemma 4.1]{HLYZ}), implies that
$$|\rho_{[h_t]}-\rho_{K}|\leq m_2\rho_{K}|\log\rho_{[h_t]}-\log\rho_{K}|\leq m_2m_3M|t|=M_0|t|,$$
on $\mathbb{S}^{n-1}$ for $t\in(-\varepsilon_0,\varepsilon_0)$. Here $M_0=m_2m_3M$ where $M$ comes from \cite[Lemma 4.1]{HLYZ}. The proof is completed.
\end{proof}

The following variational formula gives rise to the $L_p$ Gauss dual surface area measure.

\begin{theorem}\label{t3.2}
For $p\neq0$ and $q\in\mathbb{R}$. Let $K\in\mathcal{K}_o^n$,  $f\in C(\mathbb{S}^{n-1})$ and $h_t$ be given in Lemma \ref{l3.1}.
Then
$$\lim_{t\rightarrow0}\frac{\widetilde{V}_{\gamma,q}([h_t])
-\widetilde{V}_{\gamma,q}(K)}{t}=-\frac{1}{p}
\int_{\mathbb{S}^{n-1}}f(v)^p
d\widetilde{C}_{\gamma,p,q}(K,v).$$
\end{theorem}

\begin{proof}
From the definition of $h_t$, we have
$$\log h_t(u)=\log h_K(u)+\frac{tf(u)^p}{ph_K(u)^p},\ u\in\mathbb{S}^{n-1}.$$
From (\ref{1.1}), it is equivalent to
$$\widetilde{V}_{\gamma,q}([h_t])=\int_{\mathbb{S}^{n-1}}
\int_{\rho_{[h_t]}}^\infty e^{-\frac{r^2}{2}}r^{q-1}drdu.$$
Let $t\rightarrow0$. Then there exist $\lambda_1, \lambda_2>0$ such that $\rho_{[h_t]}, \rho_{K}\in[\lambda_1, \lambda_2]$. Set
$$F(s)=\int_s^\infty e^{-\frac{r^2}{2}}r^{q-1}dr.$$
According to the mean value theorem, there exists $\xi$ between $\rho_{[h_t]}$ and $\rho_{K}$ such that
\begin{align}\label{3.4}
|F(\rho_{[h_t]})-F(\rho_{K})|=|F^\prime(\xi)|
|\rho_{[h_t]}-\rho_{K}|\leq e^{-\frac{\xi^2}{2}}\xi^{q-1}M_0|t|\leq\lambda_3|t|,
\end{align}
where $\lambda_3=\max_{\xi\in[\lambda_1,\lambda_2]}
(e^{-\frac{\xi^2}{2}}\xi^{q-1}M_0)$, and $M_0$ comes from Lemma \ref{3.1}.

By Lemma \ref{3.1}, (\ref{3.4}) and \cite[(3.39)]{LYZ}, employing the dominated convergence theorem, one has

\begin{align*}
&\lim_{t\rightarrow0}\frac{\widetilde{V}_{\gamma,q}([h_t])-
\widetilde{V}_{\gamma,q}(K)}{t}
=\lim_{t\rightarrow0}\int_{\mathbb{S}^{n-1}}
\frac{F(\rho_{[h_t]})-F(\rho_{K})}{t}du\\
&=-\frac{1}{p}\int_{\mathbb{S}^{n-1}}e^{-\frac{\rho_K^2(u)}{2}}
\rho_K^{q}(u)\frac{f(\alpha_K(u))^p}{h_K(\alpha_K(u))^p}du\\
&=-\frac{1}{p}\int_{\mathbb{S}^{n-1}}f(\alpha_K(u))^p
e^{-\frac{\rho_K^2(u)}{2}}\rho_K^{q-n}(u)\frac{\rho_K^{n}(u)}
{h_K(\alpha_K(u))^p}du\\
&=-\frac{1}{p}\int_{\partial^\prime K}f(\nu_K(x))^p
e^{-\frac{|x|^2}{2}}|x|^{q-n}(x\cdot\nu_K(x))^{1-p}
d\mathcal{H}^{n-1}(x)\\
&=-\frac{1}{p}\int_{\mathbb{S}^{n-1}}f(v)^p
d\widetilde{C}_{\gamma,p,q}(K,v).
\end{align*}
The proof is completed.
\end{proof}

Thus, the definition of $L_p$-Gaussian dual curvature measure deduced from the variational formula in Theoren \ref{t3.2} is obtained.

\begin{definition}\label{d3.3}
Let $K\in\mathcal{K}_o^n$ and $p,q\in\mathbb{R}$. Define the $L_p$-Gaussian dual curvature measure of $K$ by: for Borel set $\omega\subset\mathbb{S}^{n-1}$,
$$\widetilde{C}_{\gamma,p,q}(K,\omega)
=\int_{\nu_K^{-1}(\omega)}(x\cdot\nu_K(x))^{1-p}
e^{-\frac{|x|^2}{2}}|x|^{q-n}d\mathcal{H}^{n-1}(x).$$
\end{definition}

The following result gives the weak convergence of the $L_p$-Gaussian dual curvature measure.

\begin{proposition}\label{p3.4}
If $K_i\in\mathcal{K}_o^n$ such that $K_i$ converges to $K_0\in\mathcal{K}_o^n$ in Hausdorff metric, then
$\widetilde{C}_{\gamma,p,q}(K_i,\cdot)$ converges to $\widetilde{C}_{\gamma,p,q}(K_0,\cdot)$ weakly.
\end{proposition}

\begin{proof}
Let $f: \omega\rightarrow\mathbb{R}$ be a continuous function. Suppose that $K_i\rightarrow K_0$ with respect to Hausdorff metric. Then
$\rho_{K_i}\rightarrow\rho_{K_0}$ uniformly on $\mathbb{S}^{n-1}$, and $\alpha_{K_i}\rightarrow\alpha_{K_0}$ almost everywhere on $\mathbb{S}^{n-1}$ with respect to spherical Lebesgue measure (by \cite[Lemma 2.2]{HLYZ}).
Moreover, $h_{K_i}(\alpha_{K_i},\cdot)\rightarrow
h_{K_0}(\alpha_{K_0},\cdot)$ uniformly on $\mathbb{S}^{n-1}$.
Since $K_0\in\mathcal{K}_o^n$, then there exists a constant $C>0$ such that $\frac{1}{C}B\subset K_i,K_0\subset CB$ for
sufficiently large $i$.
Thus
\begin{align*}
e^{-\frac{\rho_{K_i}^2(u)}{2}}\rho_{K_i}^q(u)
\frac{f(\alpha_{K_i}(u))^p}{h_{K_i}(\alpha_{K_i}(u)^p}
\rightarrow e^{-\frac{\rho_{K_0}^2(u)}{2}}\rho_{K_0}^q(u)
\frac{f(\alpha_{K_0}(u))^p}{h_{K_0}(\alpha_{K_0}(u)^p}
\end{align*}
uniformly on $\mathbb{S}^{n-1}$.

By Definition \ref{d3.3} and Lebesgue's dominated convergence theorem, one has
\begin{align*}
&\int_{\mathbb{S}^{n-1}}f(v)^pd\widetilde{C}_{\gamma,p,q}(K_i,v)\\
&=\int_{\partial^\prime {K_i}}f(\nu_{K_i}(x))^p(x\cdot\nu_{K_i}(x))^{1-p}
e^{-\frac{|x|^2}{2}}|x|^{q-n}d\mathcal{H}^{n-1}(x)\\
&=\int_{\mathbb{S}^{n-1}}f(\alpha_{K_i}(u))^p
h_{K_i}(\alpha_{K_i}(u))^{-p}
e^{-\frac{\rho_{K_i}^2(u)}{2}}\rho_{K_i}^q(u)du\\
&\rightarrow\int_{\mathbb{S}^{n-1}}f(\alpha_{K_0}(u))^p
h_{K_0}(\alpha_{K_0}(u))^{-p}
e^{-\frac{\rho_{K_0}^2(u)}{2}}\rho_{K_0}^q(u)du\\
&=\int_{\mathbb{S}^{n-1}}f(\alpha_{K_0}(u))^p
e^{-\frac{\rho_{K_0}^2(u)}{2}}\rho_{K_0}(u)^{q-n}
\frac{\rho_{K_0}(u)^n}
{h_{K_0}(\alpha_{K_0}(u))^{p}}du\\
&=\int_{\partial^\prime {K_0}}f(\nu_{K_0}(x))^p
(x\cdot\nu_{K_0}(x))^{1-p}
e^{-\frac{|x|^2}{2}}|x|^{q-n}d\mathcal{H}^{n-1}(x)\\
&=\int_{\mathbb{S}^{n-1}}f(v)^p
d\widetilde{C}_{\gamma,p,q}({K_0},v).
\end{align*}
The proof is completed.
\end{proof}

\begin{proposition}\label{p3.5}
For $p\neq0$, $q\in\mathbb{R}$ and $K\in\mathcal{K}_o^n$, the $L_p$-Gauss dual curvature measure $\widetilde{C}_{\gamma,p,q}(K,\cdot)$ is absolutely continuous with respect to the surface area measure $S(K,\cdot)$.
\end{proposition}

\begin{proof}
For any Borel set $\omega\subset\mathbb{S}^{n-1}$ and $K\in\mathcal{K}_o^n$, the surface area measure $S(K,\cdot)$ of $K$ is defined as (see \cite{S})
$$S(K,\omega)=\mathcal{H}^{n-1}(\nu_K^{-1}(\omega)).$$

Let $S(K,\omega)=0$. It is equivalent to $\mathcal{H}^{n-1}(\nu_K^{-1}(\omega))=0$.
Since we are integrating over a set of measure zero, it follows that
$$\widetilde{C}_{\gamma,p,q}(K,\omega)
=\int_{\nu_K^{-1}(\omega)}(x\cdot\nu_K(x))^{1-p}
e^{-\frac{|x|^2}{2}}|x|^{q-n}d\mathcal{H}^{n-1}(x)=0.$$
This implies that
$\widetilde{C}_{\gamma,p,q}(K,\cdot)$ is absolutely continuous with respect to the surface area measure $S(K,\cdot)$.
\end{proof}

{\bf The $L_p$-Gauss dual Minkowski problem}.
For $p,q\in\mathbb{R}$,
under what necessary and/or sufficient condition on a non-zero finite Borel measure $\mu$ on $\mathbb{S}^{n-1}$ does there exist a set $K\in\mathcal{K}_o^n$ such that
\begin{align}\label{3.5}
\mu=\widetilde{C}_{\gamma,p,q}(K,\cdot)?
\end{align}
If $K$ exists, to what extent is it unique?

As we all know, if $\partial K$ is smooth with positive Gauss curvature, then the surface area measure $S(K,\cdot)$ of $K$ is absolutely continuous with respect to spherical Lebesgue measure $S$ and its density is
\begin{align}\label{3.6}
\frac{dS(K,\cdot)}{dS}=\det(\nabla^2h_K+h_KI).
\end{align}

Assume $K$ is strictly convex. By Proposition \ref{p3.5} and (\ref{2.5}), one has
\begin{align*}
\widetilde{C}_{\gamma,p,q}(K,\eta)&=\int_{\nu_K^{-1}(\eta)}
(x\cdot\nu_K(x))^{1-p}|x|^{q-n}
e^{-\frac{|x|^2}{2}}d\mathcal{H}^{n-1}(x)\\
&=\int_{\eta}h_K(v)^{1-p}(|\nabla h_K(v)|^2+h_K(v)^2)^{\frac{q-n}{2}}e^{-\frac{|\nabla h_K(v)|^2+h_K(v)^2}{2}}dS(K,v).
\end{align*}
Namely,
\begin{align}\label{3.7}
d\widetilde{C}_{\gamma,p,q}(K,\cdot)
=h_K^{1-p}(|\nabla h_K|^2+h_K^2)^{\frac{q-n}{2}}e^{-\frac{|\nabla h_K|^2+h_K^2}{2}}dS(K,\cdot).
\end{align}
Let $\partial K$ is smooth with positive Gauss curvature for $K\in\mathcal{K}_o^n$. By (\ref{3.6}) and (\ref{3.7}), we have
$$\frac{d\widetilde{C}_{\gamma,p,q}(K,\cdot)}{dS}
=h_K^{1-p}(|\nabla h_K|^2+h_K^2)^{\frac{q-n}{2}}e^{-\frac{|\nabla h_K|^2+h_K^2}{2}}\det(\nabla^2h_K+h_KI).$$
Therefore, when the given measure $\mu$ has a density $f$ that is an integral nonnegative function on $\mathbb{S}^{n-1}$, the $L_p$-Gauss dual Minkowski problem is equivalent to solving the following Monge-Amp\`{e}re type equation
\begin{align*}
e^{-\frac{|\nabla h_K|^2+h_K^2}{2}}h_K^{1-p}
(|\nabla h_K|^2+h_K^2)^{\frac{q-n}{2}}
\det(\nabla^2h_K+h_KI)=f.
\end{align*}

\section{The variational method to generate weak solution}\label{S4}
This section aims to prove Theorem \ref{t1.2} by transforming it into an optimization problem.
\subsection{An associated optimization problem}
Let $\mathcal{K}_e^n$ be the set of origin-symmetric convex bodies. For any non-zero finite Borel measure $\mu$ on $\mathbb{S}^{n-1}$, $p>0$ and $q>0$. Define $\Phi: \mathcal{K}_e^n\rightarrow\mathbb{R}$ by
$$\Phi(Q)=\frac{1}{p}\int_{\mathbb{S}^{n-1}}h_Q(v)^pd\mu(v)
+\int_{\mathbb{S}^{n-1}}F(\rho_Q(u))du, \ for \ Q\in\mathcal{K}_e^n,$$
where $F(\rho_Q(u))=\int_{\rho_Q(u)}^{\infty}
e^{-\frac{r^2}{2}}r^{q-1}dr$.
We consider the minimization problem as follows:
\begin{align}\label{4.0}
\inf\{\Phi(Q):Q\in\mathcal{K}_e^n\}.
\end{align}

\begin{lemma}\label{l4.1}
Let $\mu$ be a non-zero finite Borel measure on $\mathbb{S}^{n-1}$. If $p>0, q>0$ and $K\in\mathcal{K}_e^n$ satisfies
\begin{align}\label{4.1}
\Phi(K)=\inf\{\Phi(Q):Q\in\mathcal{K}_e^n\},
\end{align}
then
$$\mu=\widetilde{C}_{\gamma,p,q}(K,\cdot).$$
\end{lemma}

\begin{proof}
For each $f\in C^+_e(\mathbb{S}^{n-1})$, define a functional $\phi: C^+_e(\mathbb{S}^{n-1})\rightarrow\mathbb{R}$ by
$$\phi(f)=\frac{1}{p}\int_{\mathbb{S}^{n-1}}f(v)^pd\mu(v)
+\int_{\mathbb{S}^{n-1}}F(\rho_{[f]})du.$$

We claim that for each $f\in C^+_e(\mathbb{S}^{n-1})$
\begin{align}\label{4.2}
\phi(f)\leq\phi(h_K).
\end{align}
Since $h_{[f]}\leq f$ and $[h_K]=K$, one has
$$\phi(f)\leq\phi(h_{[f]})=\Phi([f])\leq \Phi(K)=\phi(h_k),$$
where the second inequality sign follows from (\ref{4.1}).

For any $g\in C(\mathbb{S}^{n-1})$ and $t\in(-\varepsilon,\varepsilon)$ where $\varepsilon>0$ is sufficiently small. Let
$$h_t(v)=h_K(v)e^{tg(v)}.$$
By (\ref{4.2}), since $h_K$ is a minimizer of (\ref{4.1}), the Lagrange multiplier method says,
$$\frac{d}{dt}\phi(h_t)|_{t=0}=0.$$
This, together with the formula in \cite[Lemma 3.1]{F}, implies
$$\int_{\mathbb{S}^{n-1}}h_K(v)^pg(v)d\mu(v)
=\int_{\mathbb{S}^{n-1}}g(v)d\widetilde{C}_{\gamma,q}(K,v).$$
Thus
$$d\mu(v)=h_K(v)^{-p}d\widetilde{C}_{\gamma,q}(K,v).$$
We conclude that $\mu=\widetilde{C}_{\gamma,p,q}(K,\cdot)$.
\end{proof}

\subsection{Existence of an optimizer}

\begin{lemma}\label{l4.2}
Let $\mu$ be a non-zero finite even Borel measure on $\mathbb{S}^{n-1}$ and not concentrated in any closed hemisphere. If $p>0$ and $q>0$, then there exists a convex body $K\in\mathcal{K}_e^n$ such that
\begin{align*}
\Phi(K)=\inf\{\Phi(Q):Q\in\mathcal{K}_e^n\}.
\end{align*}
\end{lemma}

\begin{proof}
Suppose $Q_l$ is a sequence of origin-symmetric convex
bodies and
\begin{align}\label{4.3}
\lim_{l\rightarrow\infty}\Phi(Q_l)
=\inf\{\Phi(Q):Q\in\mathcal{K}_e^n\}.
\end{align}

We claim that $Q_l$ is uniformly bounded. If not, then there exists $u_l\in\mathbb{S}^{n-1}$ such that $\rho_{Q_l}(u_l)\rightarrow\infty$ as $l\rightarrow\infty$. By the definition of support function, one has
$$\rho_{Q_l}(u_l)(v\cdot u_l)_+\leq h_{Q_l}(v).$$
From the definitions of $\Phi$ and Gamma function, we have
\begin{align*}
\Phi(Q_l)&=\frac{1}{p}\int_{\mathbb{S}^{n-1}}h_{Q_l}(v)^pd\mu(v)
+n\omega_n\int_{\rho_{Q_l}(v)}^\infty e^{-\frac{r^2}{2}}r^{q-1}dr\\
&\geq \frac{1}{p}\int_{\mathbb{S}^{n-1}}
\rho_{Q_l}(u_l)^p(v\cdot u_l)_+^pd\mu(v)
+n\omega_n\frac{\sqrt{2}^q}{2}
\int_{\frac{\rho_{Q_l}(v)}{2}}^\infty
e^{-r}r^{\frac{q}{2}-1}dr\\
&=\frac{\rho_{Q_l}(u_l)^p}{p}\int_{\mathbb{S}^{n-1}}
(v\cdot u_l)_+^pd\mu(v)+n\omega_n
\frac{\sqrt{2}^q}{2}\Gamma(\frac{q}{2})\\
&\geq\frac{\rho_{Q_l}(u_l)^p}{p}\int_{\mathbb{S}^{n-1}}
(v\cdot u_l)_+^pd\mu(v)+n\omega_n,
\end{align*}
where the last inequality follows from the property of Gamma function $\Gamma$, i.e., $\frac{\sqrt{2}^q}{2}\Gamma(\frac{q}{2})\geq1$, and $(t)_+=\max\{t,0\}$ for any $t\in\mathbb{R}$.
By the fact that $\mu$ is not concentrated in any closed hemisphere, we may find $c_0>0$ such that
$$\int_{\mathbb{S}^{n-1}}
(v\cdot u_l)_+^pd\mu(v)>c_0.$$
Therefore
$$\Phi(Q_l)>\frac{c_0\rho_{Q_l}(u_l)^p}{p}
+n\omega_n\rightarrow\infty$$
as $l \rightarrow\infty$. But this contradicts (\ref{4.3}). So we conclude that $Q_l$ is uniformly bounded. By the Blaschke selection theorem in \cite{S}, there is a convergent subsequence of $Q_l$, still denoted by $Q_l$, converges to a compact convex set $K$ of $\mathbb{R}^n$.

Since $K$ is origin-symmetric, this means that $K$ contains the origin as its interior point. Therefore, the convex body $K$ is a minimizer to the optimization (\ref{4.0}).
\end{proof}

\subsection{Existence of symmetric solution to $L_p$-Gaussian dual  Minkowski problem}
Combining Lemma \ref{l4.1} and Lemma \ref{l4.2}, we get the existence Theorem \ref{t1.2} for the $L_p$-Gaussian dual Minkowski problem stated in the introduction.

\begin{theorem}\label{t4.3}
Let $\mu$ be a non-zero finite even Borel measure on $\mathbb{S}^{n-1}$ and not concentrated in any closed hemisphere. If $p>0$ and $q>0$, then there exists a convex body $K\in\mathcal{K}_e^n$ such that
\begin{align*}
\mu=\widetilde{C}_{\gamma,p,q}(K,\cdot).
\end{align*}
\end{theorem}

\section{The curvature flow method to generate smooth  solution}\label{S5}
In this section, the existence of smooth solution to the $L_p$-Gauss dual Minkowski problem will be obtained by curvature flow method.

\subsection{Gauss curvature flow and its related geometric function}

We consider a class of Gauss curvature flow of a family of closed convex hypersurfaces $\{\Omega_t\}$ given by
$\Omega_t=F(\mathbb{S}^{n-1},t)$,
where $F: \mathbb{S}^{n-1}\times [0,T)\rightarrow\mathbb{R}^n$ is the smooth map that satisfies
\begin{align}\label{5.1}
\begin{cases}
\frac{\partial X(x,t)}{\partial t}=-f(x)(x\cdot F(x,t))^{p}e^{\frac{|F(x,t)|^2}{2}}
|F(x,t)|^{n-q}\mathcal{K}(x,t)\nu+F(x,t),\\
F(x,0)=F_0(x),
\end{cases}
\end{align}
where $f$ is a given positive smooth function on $\mathbb{S}^{n-1}$, $\mathcal{K}$ is the Gauss curvature of $\Omega_t$ at $F(x,t)$, $\nu$ is the out normal of $\Omega_t$ at $F(x,t)$,
and $T$ is the maximal time for which the solution of (\ref{5.1}) exists.

As discussed in Section 2, the support function of convex hypersurface $\Omega_t$ can be expressed as $h(x,t)=x\cdot F(x,t)$, we thus derive the evolution equation for $h(\cdot,t)$ from the flow (\ref{5.1}) as follows
\begin{align}\label{5.2}
\begin{cases}
\frac{\partial h(x,t)}{\partial t}
=-f(x)h^{p}e^{\frac{|\nabla h|^2+h^2}{2}}
(|\nabla h|^2+h^2)^{\frac{n-q}{2}}\mathcal{K}
+h(x,t),\\
h(x,0)=h_0(x),
\end{cases}
\end{align}

From (\ref{2.5}) and the definition of radial function, there is
\begin{align}\label{5.3}
\rho(u)u=\nabla h(x)+h(x)x.
\end{align}
Let $x=x(u,t)$, we have
$$\log\rho(u,t)=\log h(x,t)-\log x\cdot u.$$
Differentiating the above identity, we obtain
\begin{align}\label{5.4}
\frac{\partial\rho(u,t)}{\partial t}
=\frac{\rho(u,t)}{h(x,t)}\frac{\partial h(x,t)}{\partial t}.
\end{align}
By (\ref{5.2}), (\ref{5.3}) and (\ref{5.4}), the evolution equation of $\rho(u,t)$ can be described as
\begin{align}\label{5.5}
\begin{cases}
\frac{\partial\rho(u,t)}{\partial t}
=-f(\xi)h^{p-1}\rho^{n-q+1}
e^{\frac{\rho^2}{2}}\mathcal{K}+\rho(u,t),\\
\rho(u,0)=\rho_0(u).
\end{cases}
\end{align}

Now we investigate the characteristic of a more important geometric functional that are key to proving the long-time existence of solutions to the flow (\ref{5.1}). Firstly, we define
$$\phi(r)=\int^{r}s^{q-1}e^{-\frac{s^2}{2}}ds,$$
and
$$\varphi(t)=\int^t\frac{1}{s^{1-p}}ds.$$

\begin{lemma}\label{l5.1}
For convex hypersurface $\Omega_t$, we define the functional
$$\Phi(\Omega_t)=\int_{\mathbb{S}^{n-1}}f(x)
\varphi(h(x,t))dx-\int_{\mathbb{S}^{n-1}}\phi(\rho(u,t))du.$$
Let $p,q\in\mathbb{R}$. Then $\Phi(\Omega_t)$   is monotone non-increasing along the flow (\ref{5.1}). That is
$$\frac{\partial}{\partial t}\Phi(\Omega_t)\leqslant0,$$
with equality if and only if the support function of $\Omega_t$ satisfies the equation (\ref{1.2}).
\end{lemma}

\begin{proof}
By (\ref{5.2}) and (\ref{5.4}), it follows from $du=\frac{h}{\rho^n\mathcal{K}}dx$ that
\begin{align*}
\partial_t\Phi(\Omega_t)&=\int_{\mathbb{S}^{n-1}}
\frac{f(x)}{h^{1-p}}\partial_thdx-\int_{\mathbb{S}^{n-1}}
\rho^{q-1}e^{-\frac{\rho^2}{2}}\partial_t\rho du\\
&=\int_{\mathbb{S}^{n-1}}\frac{f(x)}{h^{1-p}}\partial_thdx
-\int_{\mathbb{S}^{n-1}}
\frac{\rho^{q-n}e^{-\frac{\rho^2}{2}}}
{\mathcal{K}}\partial_thdx\\
&=\int_{\mathbb{S}^{n-1}}\frac{\rho^{q-n}e^{-\frac{\rho^2}{2}}}
{\mathcal{K}}\bigg(\frac{f}{h^{1-p}}\frac{\mathcal{K}} {\rho^{q-n}e^{-\frac{\rho^2}{2}}}-1\bigg)\partial_thdx\\
&=-\int_{\mathbb{S}^{n-1}}\frac{h\rho^{q-n}e^{-\frac{\rho^2}{2}}}
{\mathcal{K}}\bigg(\frac{f}{h^{1-p}}\frac{\mathcal{K}} {\rho^{q-n}e^{-\frac{\rho^2}{2}}}-1\bigg)^2dx\\
&\leq0.
\end{align*}
Equality hold if and only if $$\frac{f}{h^{1-p}}\frac{\mathcal{K}} {\rho^{q-n}e^{-\frac{\rho^2}{2}}}=1,$$
that is the equation (\ref{1.2}).
\end{proof}

\subsection{A priori estimates}

\subsubsection{$C^0,C^1$-estimates}

\begin{lemma}\label{l5.2}
Let $h(\cdot,t)$ be a smooth solution of (\ref{5.2}). Assume that $f$ is a positive smooth function on $\mathbb{S}^{n-1}$ satisfying (\ref{1.3}) and $p, q\in\mathbb{R}$. Then there exist positive constants c and C independent of $t$ such that
\begin{align}\label{5.6}
c\leqslant h(\cdot,t)\leqslant C \ \ and \ \
c\leqslant \rho(\cdot,t)\leqslant C, \ \forall
 (\cdot,t)\in \mathbb{S}^{n-1}\times [0,T).
\end{align}
\end{lemma}

\begin{proof}
From the definitions of support and radial functions, we only need to prove one of (\ref{5.6}).

Suppose the spatial maximum of the $h(\cdot,t)$ is attained at $x_1\in\mathbb{S}^{n-1}$. Then, from (\ref{5.3}), at $x_1$
$$\nabla h=0, \ \nabla^2 h\leqslant0\ \ and \ \ h=\rho.$$
By (\ref{5.2}), one has
\begin{align*}
\partial_th&\leqslant-f(x)h^{p-q+1}e^{\frac{h^2}{2}}+h\\
&=h^{p-q+1}e^{\frac{h^2}{2}}
(h^{q-p}e^{-\frac{h^2}{2}}-f(x)).
\end{align*}
Taking $A=\limsup_{s\rightarrow+\infty}(s^{q-p}
e^{-\frac{s^2}{2}})$. By (\ref{1.3}), $\epsilon=\frac{1}{2}(\min_{\mathbb{S}^{n-1}}f(x)-A)$ is positive and there exists some positive constant $c_1$ such that
$$h^{q-p}e^{-\frac{h^2}{2}}<A+\epsilon \ for \ h>c_1.$$
This, together with (\ref{1.3}), yields
$$h^{q-p}e^{-\frac{h^2}{2}}-f(x)
<A+\epsilon-\min_{\mathbb{S}^{n-1}}f(x)<0,$$
which implies that at $x_1$
$$\partial_th<0.$$
Therefore
$$h\leqslant\max\{c_1,\max_{\mathbb{S}^{n-1}}h(x,0)\}.$$

Similarly, we can estimate the spatial minimum of $h(\cdot,t)$.
Suppose the spatial minimum of the $h(\cdot,t)$ is attained at $x_2\in\mathbb{S}^{n-1}$. Then, from (\ref{5.3}), at $x_2$
$$\nabla h=0, \ \nabla^2 h\geqslant0\ \ and \ \ h=\rho.$$
By (\ref{5.2}), one has
\begin{align*}
\partial_th&\geq-f(x)h^{p-q+1}e^{\frac{h^2}{2}}+h\\
&=h^{p-q+1}e^{\frac{h^2}{2}}
(h^{q-p}e^{-\frac{h^2}{2}}-f(x)).
\end{align*}
Taking $a=\liminf_{s\rightarrow0^+}(s^{q-p}
e^{-\frac{s^2}{2}})$. By (\ref{1.3}), $\epsilon=\frac{1}{2}(a-\max_{\mathbb{S}^{n-1}}f(x))$ is positive and there exists some positive constant $c_2$ such that
$$h^{q-p}e^{-\frac{h^2}{2}}>a-\epsilon \ for \ h<c_2.$$
This, together with (\ref{1.3}), yields
$$h^{q-p}e^{-\frac{h^2}{2}}-f(x)
>a-\epsilon-\max_{\mathbb{S}^{n-1}}f(x)>0,$$
which implies that at $x_2$
$$\partial_th>0.$$
Therefore
$$h\geqslant\min\{c_2,\min_{\mathbb{S}^{n-1}}h(x,0)\}.$$
The proof is completed.
\end{proof}

\begin{lemma}\label{l5.3}
Under the assumptions of Lemma \ref{l5.2}, then we have
\begin{align*}
|\nabla h(\cdot,t)|\leqslant C \  and \
|\nabla\rho(\cdot,t)|\leqslant C,  \  \forall
 (\cdot,t)\in \mathbb{S}^{n-1}\times [0,T),
\end{align*}
where $C$ is a positive constant as in Lemma \ref{l5.2}.
\end{lemma}

\begin{proof}
By (\ref{2.5}) and (\ref{5.3}), we know that
$$\min_{\mathbb{S}^{n-1}}h(\cdot,t)\leqslant\rho(\cdot,t)\leqslant\max_{\mathbb{S}^{n-1}}h(\cdot,t),$$
and
$$\rho^2=h^2+|\nabla h|^2.$$
These facts together with Lemma \ref{l5.2} imply the result.
\end{proof}

\subsubsection{$C^2$-estimates}
In this subsection, we will establish the upper and lower bounds of the principal curvatures. These estimates can be obtained by considering proper auxiliary functions, see e.g., \cite{LL} for similar techniques.

\begin{lemma}\label{l5.4}
Let $h(\cdot,t)$ be a smooth solution of (\ref{5.2}). Assume that $f$ is a positive smooth function on $\mathbb{S}^{n-1}$ satisfying (\ref{1.3}). Then there exist positive constants c and C, independent of $t$, such that the principal curvature $\kappa_i$, $i=1,...,n-1$, of $\Omega_t$ satisfying
\begin{align}\label{5.7}
c\leqslant \kappa_i\leqslant C, \ \forall
 (\cdot,t)\in \mathbb{S}^{n-1}\times [0,T).
\end{align}
\end{lemma}

\begin{proof}
{\bf Step 1:} We need to prove the upper bound of Gauss curvature $\mathcal{K}$. It is essential to
construct the following auxiliary function
\begin{align}\label{5.8}
\Theta(x,t)=\frac{-\partial_th+h}{h-\varepsilon_0}
=\frac{f(x)h^p\rho^{n-q}e^{\frac{\rho^2}{2}}
\mathcal{K}}{h-\varepsilon_0},
\end{align}
where $\varepsilon_0$ is a positive constant satisfying $\varepsilon_0=\frac{1}{2}\min h(\cdot,t)$, $\forall(\cdot,t)\in \mathbb{S}^{n-1}\times[0,T)$.

From (\ref{5.8}), the upper bound of $\mathcal{K}$ follows from $\Theta(\cdot,t)$. Hence we only need to derive the upper bound of $\Theta(x,t)$.
Suppose the spatial maximum of $\Theta$ is obtained at $x_0$. Then
\begin{align}\label{5.9}
0=\nabla_i\Theta=\frac{-\partial_th_i+h_i}{h-\varepsilon_0}
+\frac{(\partial_th-h)h_i}{(h-\varepsilon_0)^2},
\end{align}
and using (\ref{5.9}), we get
\begin{align}\label{5.10}
0\geq\nabla_{ij}\Theta=\frac{-\partial_th_{ij}
+h_{ij}}{h-\varepsilon_0}+\frac{(\partial_th-h)
h_{ij}}{(h-\varepsilon_0)^2},
\end{align}
where $\nabla_{ij}\Theta\leq0$ should be understood in the sense of negative semi-definite matrix.

{As in the background metrial, namely Section \ref{S2}, we know the fact $b_{ij}=h_{ij}+h\delta_{ij}$, and $b^{ij}$ its inverse matrix.} This, together with (\ref{5.10}), yields
\begin{align*}
\partial_tb_{ij}&=\partial_th_{ij}+\partial_th\delta_{ij}\\
&\geq h_{ij}+\frac{(\partial_th-h)h_{ij}}{h-\varepsilon_0}
+\partial_th\delta_{ij}\\
&=b_{ij}+\frac{(\partial_th-h)h_{ij}}{h-\varepsilon_0}
+(\partial_th-h)\delta_{ij}\\
&=b_{ij}+\frac{\partial_th-h}{h-\varepsilon_0}(h_{ij}
+h\delta_{ij}-\varepsilon_0\delta_{ij})\\
&=b_{ij}-\Theta(b_{ij}-\varepsilon_0\delta_{ij}).
\end{align*}
By the definition of Gauss curvature, we obtain
\begin{align}\label{5.11}
\nonumber \partial_t\mathcal{K}&=-\mathcal{K}b^{ij}\partial_tb_{ij}\\
&\leq-\mathcal{K}b^{ij}[b_{ij}-\Theta(b_{ij}
-\varepsilon_0\delta_{ij})]\\
\nonumber &=-\mathcal{K}[(n-1)(1-\Theta)+\Theta\varepsilon_0\mathcal{H}],
\end{align}
where $\mathcal{H}$ denotes the mean curvature of $F(\cdot,t)$.

From (\ref{5.8}) and Lemma \ref{l5.2}, there exists a positive constant $c_1$ such that
\begin{align}\label{5.12}
\frac{1}{c_1}\Theta(\cdot,t)\leq \mathcal{K}(\cdot,t)\leq c_1\Theta(\cdot,t).
\end{align}
Noting
\begin{align*}
\frac{1}{n-1}\mathcal{H}\geq \mathcal{K}^{\frac{1}{n-1}},
\end{align*}
and combining (\ref{5.11}) and (\ref{5.12}), we obtain
\begin{align}\label{5.13}
\partial_t\mathcal{K}&\leq(n-1)c_1^{-1}\Theta^2
-(n-1)c_1^{-1}\varepsilon_0\Theta^{\frac{2n-1}{n-1}}.
\end{align}

From (\ref{5.8}), we have
\begin{align*}
\partial_t\Theta=f(x)\mathcal{K}\partial_t
\bigg(\frac{h^p\rho^{n-q}e^{\frac{\rho^2}{2}}}
{h-\varepsilon_0}\bigg)
+\frac{f(x)h^p\rho^{n-q}e^{\frac{\rho^2}{2}}}
{h-\varepsilon_0}\partial_t\mathcal{K},
\end{align*}
where
\begin{align*}
\partial_t\bigg(\frac{h^p\rho^{n-q}e^{\frac{\rho^2}{2}}}
{h-\varepsilon_0}\bigg)
&=\bigg[ph^{p-1}\rho^{n-q}e^{\frac{\rho^2}{2}}\partial_th
+(n-q)\rho^{n-q-1}h^pe^{\frac{\rho^2}{2}}\partial_t\rho\\
&\ \ \ +h^p\rho^{n-q+1}e^{\frac{\rho^2}{2}}
\partial_t\rho\bigg]\frac{1}{h-\varepsilon_0}
-\frac{h^p\rho^{n-q}e^{\frac{\rho^2}{2}}\partial_th}
{(h-\varepsilon_0)^2}\\
&=(\alpha+\rho^2)\frac{h^{p-1}\rho^{n-q}
e^{\frac{\rho^2}{2}}\partial_th}{h-\varepsilon_0}
-\frac{h^p\rho^{n-q}e^{\frac{\rho^2}{2}}\partial_th}
{(h-\varepsilon_0)^2}\\
&=\bigg[(\alpha+\rho^2)h^{p-1}\rho^{n-q}e^{\frac{\rho^2}{2}}
-\frac{h^p\rho^{n-q}e^{\frac{\rho^2}{2}}}
{h-\varepsilon_0}\bigg]
\frac{h-(h-\varepsilon_0)\Theta}{h-\varepsilon_0},
\end{align*}
where $\alpha=n+p-q$.
By (\ref{5.8}), (\ref{5.12}) and (\ref{5.13}), we have at $x_0$
\begin{align}\label{5.14}
\nonumber\partial_t\Theta&\leq c_1f \bigg[(\alpha+\rho^2)h^{p-1}\rho^{n-q}e^{\frac{\rho^2}{2}}
-\frac{h^p\rho^{n-q}e^{\frac{\rho^2}{2}}}
{h-\varepsilon_0}\bigg]
\frac{(h-(h-\varepsilon_0)\Theta)\Theta}{h-\varepsilon_0}\\
&\nonumber\ \ \ +\frac{(n-1)fh^p\rho^{n-q}
e^{\frac{\rho^2}{2}}}{c_1(h-\varepsilon_0)}
(\Theta^2-\varepsilon_0\Theta^{\frac{2n-1}{n-1}})\\
&=\frac{c_1fh}{h-\varepsilon_0}\bigg[(\alpha+\rho^2)
h^{p-1}\rho^{n-q}e^{\frac{\rho^2}{2}}
-\frac{h^p\rho^{n-q}e^{\frac{\rho^2}{2}}}
{h-\varepsilon_0}\bigg]\Theta\\
&\nonumber\ \ \ -c_1f\bigg[(\alpha+\rho^2)
h^{p-1}\rho^{n-q}e^{\frac{\rho^2}{2}}
-\frac{h^p\rho^{n-q}e^{\frac{\rho^2}{2}}}
{h-\varepsilon_0}\bigg]\Theta^2\\
&\nonumber\ \ \ +\frac{(n-1)fh^p\rho^{n-q}
e^{\frac{\rho^2}{2}}}{c_1(h-\varepsilon_0)}
(\Theta^2-\varepsilon_0\Theta^{\frac{2n-1}{n-1}}).
\end{align}
Then by the definition of $\varepsilon_0$, there exists positive constants $c_2, c_3$ and $c_4$, depending only on $\min_{\mathbb{S}^{n-1}\times[0,T)}h$, $\min_{\mathbb{S}^{n-1}\times[0,T)}\rho$,
$\min_{\mathbb{S}^{n-1}}f$ and $\varepsilon_0$, such that
$$c_1f\bigg[(\alpha+\rho^2)
h^{p-1}\rho^{n-q}e^{\frac{\rho^2}{2}}
-\frac{h^p\rho^{n-q}e^{\frac{\rho^2}{2}}}
{h-\varepsilon_0}\bigg]\leq c_2,$$

$$\frac{(n-1)fh^p\rho^{n-q}
e^{\frac{\rho^2}{2}}}{c_1(h-\varepsilon_0)}\leq c_3,\ \ and \ \ \frac{h}{h-\varepsilon_0}\leq c_4.$$
Therefore (\ref{5.14}) can be further estimated as
\begin{align*}
\partial_t\Theta&\leq c_2c_4\Theta-c_2\Theta^2
+c_3(\Theta^2-\varepsilon_0\Theta^{\frac{2n-1}{n-1}})\\
&=\Theta^2[c_2c_4\Theta^{-1}-(c_2-c_3)-c_3\varepsilon_0
\Theta^{\frac{1}{n-1}}].
\end{align*}
One can see that whenever $c_2\geq c_3$ and
$\Theta>(c_2c_3^{-1}\varepsilon_0^{-1})^{\frac{n-1}{n}}$, we have
\begin{align*}
\partial_t\Theta<0.
\end{align*}
This implies that $\Theta$ has a uniform upper bound.

For any $(x,t)$
$$\mathcal{K}(x,t)=\frac{(h-\varepsilon_0)\Theta(x,t)}
{fh^p\rho^{n-q}e^{\frac{\rho^2}{2}}}
\leq\frac{(h-\varepsilon_0)\Theta(x_0,t)}
{fh^p\rho^{n-q}e^{\frac{\rho^2}{2}}}\leq C.$$
Namely, $\mathcal{K}$ has a uniform upper bound.

{\bf Step 2:} Now we prove the lower bound of principal curvature. Consider the following auxiliary function
\begin{align}\label{5.15}
\mathcal{E}(x,t)=\log\lambda_{\max}(\{b_{ij}\})
-a\log h(x,t)-b|\nabla h|^2,
\end{align}
where $a,b$ are positive constants to be specified later, $\lambda_{\max}(\{b_{ij}\})$ is the maximal eigenvalue of $\{b_{ij}\}$. As showed in Section \ref{S2}, we know that the eigenvalue of $\{b_{ij}\}$ and $\{b^{ij}\}$ are respectively the principal radii and principal curvature of $\Omega_t$.

For any fixed $t\in[0,T)$, suppose that the maximum of $\mathcal{E}(x,t)$ is attained at $x_1\in\mathbb{S}^{n-1}$.
By a rotation of coordinates, we may assume that $\{b^{ij}(x_1,t)\}$ is diagonal, and $\lambda_{\max}(\{b_{ij}(x_1,t)\})=b_{11}(x_1,t)$.
Therefore, to derive a positive lower bound of principal curvatures, it is equivalent to estimate the upper bound of $b_{11}$.
Based on the above assumption, (\ref{5.15}) can be transformed into
\begin{align}\label{5.16}
\widetilde{\mathcal{E}}(x,t)=\log b_{11}
-a\log h(x,t)-b|\nabla h|^2.
\end{align}
Using the above assumption again, for any fixed $t\in[0,T)$, $\widetilde{\mathcal{E}}(x,t)$ has a local maximum at $x_1$, i.e.,
\begin{align}\label{5.17}
\nonumber0=\nabla_i\widetilde{\mathcal{E}}&=b^{11}\nabla_ib_{11}
-a\frac{h_i}{h}+2b\Sigma_jh_jh_{ji}\\
&=b^{11}\nabla_i(h_{11}+h\delta_{11})
-a\frac{h_i}{h}+2bh_ih_{ii},
\end{align}
and
\begin{align}\label{5.18}
0\geq\nabla_{ii}\widetilde{\mathcal{E}}
=b^{11}\nabla_{ii}b_{11}-(b^{11})^2
(\nabla_ib_{11})^2-a(\frac{h_{ii}}{h}-\frac{h_i^2}{h^2})
+2b(\Sigma_jh_jh_{jii}+h_{ii}^2).
\end{align}
Moreover
\begin{align}\label{5.19}
\nonumber\partial_t\widetilde{\mathcal{E}}
&=b^{11}\partial_tb_{11}-a\frac{h_t}{h}+2b\Sigma_jh_jh_{jt}\\
&=b^{11}(h_{11t}+h_t)-a\frac{h_t}{h}+2b\Sigma_jh_jh_{jt}.
\end{align}

From (\ref{5.8}), we write
\begin{align}\label{5.20}
\nonumber\log(h-h_t)&=\log(fh^p\rho^{n-q}
e^{\frac{\rho^2}{2}}\mathcal{K})\\
&=-\log\det(\nabla^2h+hI)+A(x,t),
\end{align}
where $A(x,t):=\log(fh^p\rho^{n-q}e^{\frac{\rho^2}{2}})$.

Differentiating (\ref{5.20}) with respect to $e_j$, we have
\begin{align}\label{5.21}
\frac{h_j-h_{jt}}{h-h_t}=-\sum_{i,k}
b^{ik}\nabla_jb_{ik}+\nabla_jA=-\sum_{i}b^{ii}
(h_{jii}+h_i\delta_{ij})+\nabla_jA,
\end{align}
and
\begin{align}\label{5.22}
\frac{h_{11}-h_{11t}}{h-h_t}-\frac{(h_1-h_{1t})^2}{(h-h_t)^2}
=-\sum_ib^{ii}\nabla_{11}b_{ii}+\sum_{i,k}b^{ii}b^{kk}
(\nabla_1b_{ik})^2+\nabla_{11}A.
\end{align}
By the Ricci identity on $\mathbb{S}^{n-1}$
$$\nabla_{11}b_{ij}=\nabla_{ij}b_{11}-\delta_{ij}b_{11}
+\delta_{11}b_{ij}-\delta_{1i}b_{1j}+\delta_{1j}b_{1i},$$
and (\ref{5.18}), (\ref{5.19}), (\ref{5.21}) and (\ref{5.22}), one has
\begin{align}\label{5.23}
\nonumber\frac{\partial_t\widetilde{\mathcal{E}}}{h-h_t}
&=b^{11}\bigg(\frac{h_{11t}-h_{11}+h_{11}+h-h+h_t}{h-h_t}\bigg)
-a\frac{h_t-h+h}{h(h-h_t)}+2b\frac{\Sigma_jh_jh_{jt}}{h-h_t}\\
\nonumber&=b^{11}\bigg(-\frac{(h_1-h_{1t})^2}{(h-h_t)^2}
+\sum_ib^{ii}\nabla_{11}b_{ii}-\sum_{i,k}b^{ii}b^{kk}
(\nabla_1b_{ik})^2-\nabla_{11}A\bigg)\\
\nonumber&\ \ \ +\frac{1}{h-h_t}-b^{11}+\frac{a}{h}
-\frac{a}{h-h_t}+2b\frac{\Sigma_jh_jh_{jt}}{h-h_t}\\
\nonumber&\leq b^{11}\bigg(\sum_ib^{ii}(\nabla_{ii}b_{11}
-b_{11}+b_{ii})-\sum_{i,k}b^{ii}b^{kk}(\nabla_1b_{ik})^2\bigg)\\
\nonumber&\ \ \ +\frac{1-a}{h-h_t}-b^{11}\nabla_{11}A
+\frac{a}{h}+2b\frac{\Sigma_jh_jh_{jt}}{h-h_t}\\
\nonumber&\leq\sum_ib^{ii}\bigg[(b^{11})^2(\nabla_ib_{11})^2
+a\bigg(\frac{h_{ii}}{h}-\frac{h_i^2}{h^2}\bigg)
-2b\bigg(\sum_jh_jh_{jii}+h_{ii}^2\bigg)\bigg]\\
&\ \ \ -b^{11}\sum_{i,k}b^{ii}b^{kk}(\nabla_1b_{ik})^2
-b^{11}\nabla_{11}A+\frac{a}{h}
+2b\frac{\Sigma_jh_jh_{jt}}{h-h_t}+\frac{1-a}{h-h_t}\\
\nonumber&\leq\sum_ib^{ii}a\bigg(\frac{h_{ii}+h-h}{h}
-\frac{h_i^2}{h^2}\bigg)+2b\sum_jh_j\bigg(-\sum_i
b^{ii}h_{jii}+\frac{h_{jt}}{h-h_t}\bigg)\\
\nonumber&\ \ \ -2b\sum_ib^{ii}h_{ii}^2
-b^{11}\nabla_{11}A+\frac{a}{h}+\frac{1-a}{h-h_t}\\
\nonumber&\leq-a\sum_ib^{ii}-2b\sum_ib^{ii}(b_{ii}^2-2b_{ii}h)
-b^{11}\nabla_{11}A+\frac{na}{h}+\frac{1-a}{h-h_t}\\
\nonumber&\ \ \ +2b\sum_jh_j\bigg(\frac{h_j}{h-h_t}
+\sum_ib^{ii}h_j-\nabla_jA\bigg)\\
\nonumber&\leq-b^{11}\nabla_{11}A-2b\sum_jh_j\nabla_jA
+(2b|\nabla h|^2-a)\sum_ib^{ii}-2b\sum_ib_{ii}\\
\nonumber&\ \ \ +\frac{2b|\nabla h|^2-a+1}{h-h_t}
+4b(n-1)h+\frac{na}{h}.
\end{align}

Next we calculate $-b^{11}\nabla_{11}A-2b\sum_jh_j\nabla_jA$. From the expression of $A(x,t)$, we have
\begin{align}\label{5.24}
\nabla_jA=\frac{f_j}{f}+\frac{ph_j}{h}
+\frac{(n-q+\rho^2)\rho_j}{\rho},
\end{align}
and
\begin{align}\label{5.25}
\nonumber\nabla_{11}A&=\frac{ff_{11}-f_1^2}{f^2}
+p\frac{hh_{11}-h_1^2}{h^2}\\
&\ \ \ +\frac{(n-q+\rho^2)\rho\rho_{11}
-(n-q-\rho^2)\rho_1^2}{\rho^2},
\end{align}
where
\begin{align*}
&\rho_i=\frac{hh_i+\sum h_kh_{ki}}{\rho}=\frac{h_ib_{ii}}{\rho},\\
&\rho_{ij}=\frac{hh_{ij}+h_ih_j+\sum h_kh_{kij}+\sum h_{ki}h_{kj}}{\rho}-\frac{h_ih_jb_{ii}b_{jj}}{\rho^3}.
\end{align*}
By (\ref{5.24}) and (\ref{5.25}), one has
\begin{align}\label{5.26}
\nonumber&-b^{11}\nabla_{11}A-2b\sum_jh_j\nabla_jA\\
\nonumber&=-b^{11}\frac{ff_{11}-f_1^2}{f^2}
-b^{11}p\frac{hh_{11}-h_1^2}{h^2}-b^{11}
\frac{(n-q+\rho^2)\rho_{11}}{\rho}\\
\nonumber&\ \ \ +b^{11}\frac{(n-q-\rho^2)\rho_1^2}{\rho^2}
-2b\sum_jh_j\frac{f_j}{f}
-2bp\sum_j\frac{h_j^2}{h}\\
\nonumber&\ \ \ -2b\sum_jh_j\frac{(n-q+\rho^2)\rho_j}{\rho}\\
&=-b^{11}\frac{ff_{11}-f_1^2}{f^2}-\frac{(n-1)p}{h}
+pb^{11}+b^{11}\frac{ph_1^2}{h^2}-b^{11}
\frac{(n-q+\rho^2)\rho_{11}}{\rho}\\
\nonumber&\ \ \ + (n-1)b_{11}\frac{(n-q-\rho^2)h_1^2}{\rho^4}
-2b\sum_jh_j\frac{f_j}{f}-2bp\sum_j\frac{h_j^2}{h}\\
\nonumber&\ \ \ -2b\sum_jh_j^2\frac{(n-q+\rho^2)b_{jj}}{\rho^2}\\
\nonumber&\leq C_1b^{11}+C_2b_{11}+C_3b+C_4
\end{align}
where $C_1, C_2, C_3$ and $C_4$ are positive constants, depending only on $\|f\|_{C^1(\mathbb{S}^{n-1})}$, $\|h\|_{C^1(\mathbb{S}^{n-1}\times[0,T))}$, $\min_{\mathbb{S}^{n-1}\times[0,T)}h$, $\min_{\mathbb{S}^{n-1}\times[0,T)}\rho$, and
$\min_{\mathbb{S}^{n-1}}f$.

Substituting (\ref{5.26}) into (\ref{5.23}), at $x_1$, we have
\begin{align*}
\frac{\partial_t\widetilde{\mathcal{E}}}{h-h_t}
&\leq C_1b^{11}+C_3b+C_4-(2b|\nabla h|^2-a)\sum_ib^{ii}\\
&\ \ \ -b\sum_ib_{ii}+4b(n-1)h+\frac{na}{h}<0
\end{align*}
provided $a<2b|\nabla h|^2$ and $b_{11}$ large enough.
Hence,
$$\mathcal{E}(x_1,t)=\widetilde{\mathcal{E}}(x_1,t)\leq C,$$
for some positive constant $C$, independent of $t$. The proof of Lemma \ref{l5.4} is completed.
\end{proof}

\subsection{Existence of smooth solution}
With the help {of the priori estimates} in subsection 5.2, we show that the flow (\ref{5.1}) exists for all time.

\begin{theorem}\label{t5.5}
Let {$p,q\in\mathbb{R}$}, and $f: \mathbb{S}^{n-1}\rightarrow(0,\infty)$ be a smooth, positive function satisfying
$$\limsup_{s\rightarrow+\infty}
(s^{q-p}e^{-\frac{s^2}{2}})<f<\liminf_{s\rightarrow0^+}
(s^{q-p}e^{-\frac{s^2}{2}}).$$
Then the flow (\ref{5.1}) has a smooth solution $\Omega_t=F(\mathbb{S}^{n-1},t)$ for all time $t>0$. Furthermore, a subsequence of $\Omega_t$ converges in $C^{\infty}$ to a smooth, closed, and uniformly convex hypersurface whose support function is smooth solution to the equation
\begin{align}\label{5.27}
e^{-\frac{|\nabla h|^2+h^2}{2}}h^{1-p}
(|\nabla h|^2+h^2)^{\frac{q-n}{2}}
\det(\nabla^2h+hI)=f.
\end{align}
\end{theorem}

\begin{proof}
From the $C^2$-estimates obtained in Lemma \ref{l5.4}, we know that Equation (\ref{5.1}) is uniformly parabolic on any finite time interval and has the short time existence.
By $C^0, C^1$ and $C^2$-estimates (Lemmas \ref{l5.2}, \ref{l5.3} and \ref{l5.4}), and the Krylov's theory \cite{K}, we get the H\"{o}lder continuity of $\nabla^2h$ and $\partial_th$.
Then we get the higher order derivation estimates by the regularity theory of the uniformly parabolic equations. Therefore, we obtain the long-time existence and regularity of the solution to Equation (\ref{5.1}).
Moreover, we have
\begin{align*}
\|h\|_{C_{x,t}^{i,j}(\mathbb{S}^{n-1}\times[0,T))}\leq C
\end{align*}
for some $C>0$, independent of $t$, and for each pairs of nonnegative integers $i, j$.

With the aid of the Arzel\`{a}-Ascoli theorem and a diagonal argument, there exists a sequence of $t$, denoted by $\{t_k\}_{k\in\mathbb{N}}\subset(0,\infty)$, and a smooth function $h(x)$ such that
\begin{align*}
\|h(x,t_k)-h(x)\|_{C^i(\mathbb{S}^{n-1})}\rightarrow0
\end{align*}
uniformly for any nonnegative integer $i$ as  $t_k\rightarrow\infty$.
This illustrates that $h(x)$ is a support function. Let $\Omega_{\infty}$ be a convex body determined by $h(x)$, we conclude that $\Omega_{\infty}$ is smooth and strictly convex with the origin in its interior.

In the following, we prove that Equation (\ref{5.27}) has a smooth solution.
From Lemma \ref{l5.1}, we see that
\begin{align}\label{5.28}
\partial_t\Phi(\Omega_t)\leq0.
\end{align}
If there exists a time $\tilde{t}$ such that
$$\partial_t\Phi(\Omega_t)\bigg|_{t=\tilde{t}}=0,$$
then by the equality condition in Lemma \ref{l5.1}, we have
$$e^{-\frac{|\nabla h|^2+h^2}{2}}h^{1-p}
(|\nabla h|^2+h^2)^{\frac{q-n}{2}}
\det(\nabla^2h+hI)=f,$$
namely, support function $h(x,\tilde{t})$ of $\Omega_{\tilde{t}}$ satisfies Equation (\ref{5.27}).

Next we verify the case of $\partial_t\Phi(\Omega_t)<0$.
From Lemma \ref{l5.1}, there exists a positive constant $C$, independent of $t$, such that
\begin{align}\label{5.29}
\Phi(\Omega_t)\leq C,
\end{align}
and  $\partial_t\Phi(\Omega_t)$ is uniformly continuous.
Combining (\ref{5.28}) and (\ref{5.29}), and applying the Fundamental Theorem of Calculus, we obtain
$$\int_0^t\Phi^\prime(\Omega_t)dt
=\Phi(\Omega_t)-\Phi(\Omega_0)
\leq\Phi(\Omega_t)\leq C,$$
which leads to
$$\int_0^\infty\Phi^\prime(\Omega_t)dt<C.$$
This implies that there exists a subsequence of time $t_j\rightarrow\infty$ such that
$$\lim_{t_j\rightarrow\infty}\partial_t\Phi(\Omega_{t_j})=0.$$

From the proof of Lemma \ref{l5.1}, we have
\begin{align*}
\partial_t\Phi(\Omega_t)\bigg|_{t=t_j}
=-\int_{\mathbb{S}^{n-1}}\frac{h\rho^{q-n}
e^{-\frac{\rho^2}{2}}}{\mathcal{K}}\bigg(\frac{f}{h^{1-p}}
\frac{\mathcal{K}}{\rho^{q-n}
e^{-\frac{\rho^2}{2}}}-1\bigg)^2dx\bigg|_{t=t_j}.
\end{align*}
Passing to the limit, we have
\begin{align*}
0&=\lim_{t_j\rightarrow\infty}\partial_t
\Phi(\Omega_t)\bigg|_{t=t_j}\\
&=-\int_{\mathbb{S}^{n-1}}\frac{h_{\infty}
\rho^{q-n}_{\infty}e^{-\frac{\rho^2_{\infty}}{2}}}
{\mathcal{K}_{\infty}}\bigg(\frac{f}{h_{\infty}^{1-p}}
\frac{\mathcal{K}_{\infty}}{\rho_{\infty}^{q-n}
e^{-\frac{\rho^2_{\infty}}{2}}}-1\bigg)^2dx\\
&\leqslant0,
\end{align*}
this means that
$$\frac{f}{h_{\infty}^{1-p}}\frac{\mathcal{K}_{\infty}}
{\rho_{\infty}^{q-n}e^{-\frac{\rho^2_{\infty}}{2}}}=1,$$
where $h_\infty$, $\rho_{\infty}$ and $\mathcal{K}_{\infty}$ are the support function, radial function and product of the principal curvature radii of the limit convex hypersurface $\Omega_\infty$, respectively.
The proof of Theorem \ref{t5.5} is completed.
\end{proof}

As an application, we have

\begin{corollary}\label{c5.6}
Under the assumptions of Theorem \ref{t5.5}, there exists a smooth solution to the $L_p$-Gauss dual Minkowski problem (\ref{1.2}) for $p,q\in\mathbb{R}$.
\end{corollary}

\section{Uniqueness of solution}\label{S6}
In this section, we provide a uniqueness result of the $L_p$-Gauss dual Minkowski problem under an appropriate condition, see e.g., \cite{CW,LL} for similar techniques.

\begin{theorem}\label{t6.1}
Let {$p, q\in\mathbb{R}$ with $q<p$}, and $f: \mathbb{S}^{n-1}\rightarrow(0,\infty)$ be a smooth, positive function on $\mathbb{S}^{n-1}$. Then the solution to the equation
\begin{align}\label{6.1}
e^{-\frac{|\nabla h|^2+h^2}{2}}h^{1-p}
(|\nabla h|^2+h^2)^{\frac{q-n}{2}}
\det(\nabla^2h+hI)=f
\end{align}
is unique.
\end{theorem}

\begin{proof}
Let $h_1$ and $h_2$ be two solutions to the equation (\ref{6.1}), and
$$M(x_0)=\max_{x\in\mathbb{S}^{n-1}}\frac{h_1(x)}{h_2(x)}
=\frac{h_1(x_0)}{h_2(x_0)},\ \ x_0\in\mathbb{S}^{n-1}.$$
We now prove $M(x_0)\leq1$ by contradiction. Suppose $M(x_0)>1$, then
\begin{align}\label{6.2}
h_1(x_0)>h_2(x_0).
\end{align}
At $x_0$, we have
\begin{align}\label{6.3}
0=\nabla M=\frac{h_2\nabla h_1-h_1\nabla h_2}{h_2^2}, \ \ i.e., \ \ \frac{\nabla h_1}{h_1}=\frac{\nabla h_2}{h_2},
\end{align}
and by (\ref{6.3}), one has
\begin{align}\label{6.4}
0\geq\nabla^2M=\frac{h_2\nabla^2 h_1-h_1\nabla^2h_2}{h_2^2}, \ \ i.e., \ \
\frac{\nabla^2 h_1}{h_1}-\frac{\nabla^2 h_2}{h_2}.
\end{align}

From the equation (\ref{6.1}), using (\ref{6.3}) and (\ref{6.4}), we have
\begin{align*}
1&=\frac{e^{-\frac{|\nabla h_2|^2+h_2^2}{2}}h_2^{1-p}
(|\nabla h_2|^2+h_2^2)^{\frac{q-n}{2}}
\det(\nabla^2h_2+h_2I)}
{e^{-\frac{|\nabla h_1|^2+h_1^2}{2}}h_1^{1-p}
(|\nabla h_1|^2+h_1^2)^{\frac{q-n}{2}}
\det(\nabla^2h_1+h_1I)}\\
&=\frac{e^{-\frac{(|\frac{\nabla h_2}{h_2}|^2+1)h_2^2}{2}}h_2^{1-p}[h_2^2
(|\frac{\nabla h_2}{h_2}|^2+1)]^{\frac{q-n}{2}}h_2^{n-1}
\det(\frac{\nabla^2h_2}{h_2}+I)}
{e^{-\frac{(|\frac{\nabla h_1}{h_1}|^2+1)h_1^2}{2}}h_1^{1-p}[h_1^2
(|\frac{\nabla h_1}{h_1}|^2+1)]^{\frac{q-n}{2}}h_1^{n-1}
\det(\frac{\nabla^2h_1}{h_1}+I)}\\
&=\frac{e^{-\frac{(|\frac{\nabla h_2}{h_2}|^2+1)h_2^2}{2}}h_2^{q-p}\det(\frac{\nabla^2h_2}{h_2}+I)}
{e^{-\frac{(|\frac{\nabla h_1}{h_1}|^2+1)h_1^2}{2}}h_1^{q-p}
\det(\frac{\nabla^2h_1}{h_1}+I)}\\
&\geq\frac{e^{-\frac{(|\frac{\nabla h_2}{h_2}|^2+1)h_2^2}{2}}h_2^{q-p}}
{e^{-\frac{(|\frac{\nabla h_1}{h_1}|^2+1)h_1^2}{2}}h_1^{q-p}}.
\end{align*}
Namely,
$$\bigg(\frac{e^{\frac{h_2^2}{2}}}{e^{\frac{h_1^2}{2}}}\bigg)
^{(|\frac{\nabla h_1}{h_1}|^2+1)}\geq \bigg(\frac{h_2}{h_1}\bigg)^{q-p}.$$
By (\ref{6.2}), and noting $q<p$, we have
$$\bigg(\frac{h_2}{h_1}\bigg)^{q-p}\geq1.$$
Hence
$$e^{\frac{h_2^2}{2}}\geq e^{\frac{h_1^2}{2}}.$$
That is, $h_2(x_0)\geq h_1(x_0)$. This is a contradiction to (\ref{6.2}). Thus $h_2(x)\geq h_1(x)$ for all $x\in\mathbb{S}^{n-1}$.

On the other hand, let
$$m(x_1)=\min_{x\in\mathbb{S}^{n-1}}\frac{h_1(x)}{h_2(x)}
=\frac{h_1(x_1)}{h_2(x_1)},\ \ x_1\in\mathbb{S}^{n-1}.$$
Applying the same argument as above, we can get
$$h_1(x)\geq h_2(x)$$
for all $x\in\mathbb{S}^{n-1}$. Therefore, we obtain
$$h_1(x)\equiv h_2(x)$$
for all $x\in\mathbb{S}^{n-1}$.
\end{proof}

\end{document}